\documentclass[letterpaper]{article}
\usepackage{amsmath,amssymb,latexsym,amsthm}
\usepackage{epsfig}
\usepackage{graphicx}
\usepackage{xcolor}
\usepackage{setspace}
\usepackage{fancyhdr}
\newtheorem{theorem}{Theorem}[section]
\newtheorem{lemma}[theorem]{Lemma}
\newtheorem{proposition}[theorem]{Proposition}

\newtheorem{remark}[theorem]{Remark}

\newcommand{\C}{\mathbb{C}}
\newcommand{\T}{\mathbb{T}}
\newcommand{\D}{\mathbb{D}}
\def\v<#1>{\langle #1 \rangle}

\begin{document}
\title{Hadamard-Type Asymptotics for Products of Best Rational Approximation Errors}
\author{Vasiliy A. Prokhorov}
\date{}
\newcommand{\esssup}{\mathop{\rm ess \, sup}\limits}
\def\pl{{\cal P}}
\def\mr{{\cal R}}
\def\CC{{\bf C}}
\def\pfi{{\varphi}}
\def\mm{{\cal M}}

\def\capp{\mathop{\rm cap}\nolimits}
\def\RE{\mathop{\rm Re}\nolimits}

\def\ma{{\cal A}}

\mathsurround=2pt
\parskip=1mm
\renewcommand{\thefootnote}{}

\maketitle
\pagestyle{myheadings}

\begin{abstract}
Let \(E\subset\C\) be a compact set with connected complement, and let
 \(\rho_{n,m}(f;E)\) denote the error of best uniform rational approximation to a function
\(f\) analytic on \(E\) by rational functions whose numerator and denominator
have degrees at most \(n\) and \(m\), respectively. The Saff--Gonchar theorem is the fundamental result describing the asymptotic behavior of rational approximation errors along the rows of the Walsh table. It was first proved by Saff for continua with connected complement and Jordan boundary and subsequently extended by Gonchar to regular compact sets.
Motivated by a comparison of the Saff--Gonchar theorem with Hadamard's classical theorem on Hankel determinants, we study,
for each fixed \(m\ge 0\), the asymptotic behavior as \(n\to\infty\) of the products
\[
\prod_{k=0}^{m}\rho_{n-m+k,k}(f;E).
\]
We establish Hadamard-type asymptotic formulas for these products on the closed unit disc and, more generally,
on continua with connected complement and Jordan boundary. In the disc case, our approach combines Hadamard's classical theorem and the Saff--Gonchar theorem with weighted Hankel operators and an AAK-type theorem for meromorphic approximation.
We also show that there exists a common subsequence along which the extremal exponential behavior of these products
and of the corresponding products on the closed Green sublevel sets \(E_R\) is attained.
\end{abstract}

\section{Introduction}\label{sec:introduction}
\subsection{Notation and background}

Let $E\subset\C$ be a compact set whose complement
\[
\Omega:=\widehat{\C}\setminus E
\]
is connected. We assume  that $\Omega$ is regular for the Dirichlet problem. Let $g_\Omega(z,\infty)$, $z\in\Omega$, be the Green function of $\Omega$ with pole at infinity. We extend $g_\Omega(z,\infty)$ to $E$ by setting $g_\Omega(z,\infty)=0$ on $E$; then the extended function is continuous on $\C$.

For nonnegative integers $n$ and $m$, let $\mathcal R_{n,m}$ denote the class of rational functions $r=p/q$ such that $\deg p\le n$, $\deg q\le m$, and $q\not\equiv0$. For a continuous function $f$ on $E$, set
\[
\rho_{n,m}(f;E):=\inf_{r\in\mathcal R_{n,m}}\|f-r\|_{E},
\]
where $\|\cdot\|_{E}$ denotes the uniform norm on $E$, that is,
\[
\|\varphi\|_{E}:=\max_{z\in E}|\varphi(z)|.
\]

For $R>1$, define the Green sublevel set
\[
G_R:=\{z\in\C:\ g_\Omega(z,\infty)<\log R\},
\]
and its boundary 
\[
L_R:=\{z\in\Omega:\ g_\Omega(z,\infty)=\log R\}.
\]

Let $f$ be analytic on $E$, equivalently, analytic in a neighborhood of $E$. Let $R_0>1$ be the supremum of all $R>1$ such that $f$ admits an analytic continuation to $G_R$.

For an integer $m\ge1$, let $R_m$ be the supremum of all $R>1$ such that $f$ admits a meromorphic continuation to $G_R$ with at most $m$ poles, counted with multiplicity.
Clearly,
\[
1<R_0\le R_1\le R_2\le \cdots.
\]
For fixed $m$, the sequence $\{\rho_{n,m}(f;E)\}_{n=0}^\infty$ is called the $m$th row of the Walsh table; see \cite{Walsh1969}. 
In the polynomial case \(m=0\), the classical Bernstein--Walsh theory relates the geometric rate of best polynomial approximation to the analytic continuation of the approximated function; the earliest result goes back to Bernstein \cite{Bernstein1912}; see also \cite{Walsh1969,Varga}. Saff \cite{Saff1971} was the first to prove the corresponding result for rows of the Walsh table in best rational approximation, establishing a foundational rational analogue of the classical Bernstein--Walsh theorem. For continua with connected complement and Jordan boundary, he showed that the exponential rate of approximation along a row of the Walsh table determines the radius of the maximal region of meromorphic continuation with the prescribed number of poles. Hadamard's classical theorem on Hankel determinants was a fundamental ingredient in his work. His work stimulated an important line of research connecting the rate of best rational approximation, the distribution of poles of rational approximants in the rows of the Walsh table, and the meromorphic continuation of the approximated function.
In particular, Gonchar subsequently extended Saff's theorem to regular compact sets in his paper \emph{On a theorem of Saff} \cite{GoncharSaff1974}; see also \cite{GoncharGenPade1975,GoncharPolesRows1982}. Further information on the Saff--Gonchar theorem and its extensions and on the behavior of rows of the Walsh and Pad\'e tables may be found in \cite{BGM,Su2,LevLub,Su1}.

\medskip

\noindent
\textbf{Saff--Gonchar Theorem.}
Let \(m\ge 0\) be fixed. Then
\[
\limsup_{n\to\infty}\rho_{n,m}(f;E)^{1/n}=\frac{1}{R_m}<1.
\]

\medskip
Thus, for each fixed \(m\), the exponential rate of the errors \(\rho_{n,m}(f;E)\) is determined by the radius of meromorphic continuation
of \(f\) with at most \(m\) poles. The aim of the paper is to pass beyond the asymptotics of a single row of the Walsh table and to study the products
\[
\prod_{k=0}^{m}\rho_{n-m+k,k}(f;E), \qquad n\to\infty,
\]
formed from the first \(m+1\) rows of the Walsh table. The Saff--Gonchar theorem identifies the exponential rate of a single row
with the reciprocal of the corresponding meromorphy radius \(R_m\), while Hadamard's classical theorem on Hankel determinants involves the reciprocal of the product \(R_0R_1\cdots R_m\). This suggests considering a product extension of the Saff--Gonchar theorem that is simultaneously a rational-approximation analogue of Hadamard's theorem. We obtain such an extension by taking one approximation error from each of the first \(m+1\) rows and forming their product.
We do this along the initial segment of a superdiagonal
\[
(n-m,0),\ (n-m+1,1),\ \dots,\ (n,m).
\]
We show that this product has exponential rate equal to the reciprocal of the product \(R_0R_1\cdots R_m\).
\medskip

\noindent
For the closed unit disc \(E=\overline{\D}\), this result may be viewed as a direct analogue of Hadamard's classical theorem on Hankel determinants formed from Taylor coefficients \cite{Ha}; see also \cite{Di}. We then extend the disc results to continua with connected complement and Jordan boundary through the exterior conformal map; the analytic-boundary case appears as an intermediate step in 
the argument. In the final approximation step of this extension, we follow Saff's approach from the proof of his original theorem.

\subsection{Main results}

For $R\ge 1$, we write
\[
E_R:=\{z\in\C:\ g_\Omega(z,\infty)\le \log R\}.
\]
Thus \(E_R\) is the closed Green sublevel set, and \(E_1=E\).

\medskip
\noindent
For every fixed \(m\ge 0\) and every \(R\) with \(1\le R<R_0\), the Saff--Gonchar theorem
applied to \(f\) on the Green sublevel set \(E_R\) gives
\[
\limsup_{n\to\infty}\rho_{n,m}(f;E_R)^{1/n}=\frac{R}{R_m}.
\]
We now consider the products
\[
\prod_{k=0}^{m}\rho_{n-m+k,k}(f;E_R),
\]
formed from the first \(m+1\) rows of the Walsh table. For every fixed \(m\ge 0\) and every \(R\) with \(1\le R<R_0\),
the upper estimate 
\begin{equation}\label{eq:product-upper}
\limsup_{n\to\infty} \left( \prod_{k=0}^{m}\rho_{n-m+k,k}(f;E_R) \right)^{1/n} \le \frac{R^{m+1}}{R_0R_1\cdots R_m} 
\end{equation} 
follows by applying the Saff--Gonchar theorem to each factor and multiplying the resulting estimates.

Our first result compares the products on \(E\) with the corresponding products on the larger Green sublevel sets \(E_R\), where \(1<R<R_0\).
Before stating it, we note two simple facts. First, one must exclude the degenerate case in which \(f\) belongs to the class
\(\mathcal R_{N,m}\) for some \(N\ge 0\), since in that case
\(\rho_{n,m}(f;E)=0\) for all \(n\ge N\).
Second, since \(E\subset E_R\), it follows that
\[
\rho_{n-m+k,k}(f;E)
\le
\rho_{n-m+k,k}(f;E_R),
\qquad k=0,1,\dots,m,
\]
and hence
\[
\prod_{k=0}^{m}
\frac{\rho_{n-m+k,k}(f;E)}{\rho_{n-m+k,k}(f;E_R)}
\le 1.
\]

\begin{theorem}\label{thm:first-main}
Let \(E\subset\C\) be a continuum with at least two points and connected complement, and let \(f\) be
analytic on \(E\). Fix \(m\ge0\), and assume that \(f\) is not a rational
function with at most \(m\) finite poles, counted with multiplicity.
Let \(R\) satisfy \(1<R<R_0\).
Then
\[
\limsup_{n\to\infty}
\left(
\prod_{k=0}^{m}
\frac{\rho_{n-m+k,k}(f;E)}
{\rho_{n-m+k,k}(f;E_R)}
\right)^{1/n}
\le
\frac{1}{R^{m+1}}.
\]
\end{theorem}
\medskip

\noindent
The comparison of approximation errors on \(E\) and on the sets \(E_R\)
is also related to earlier work of the author \cite{Prokhorov2004,Prokhorov2005}
and to joint work of Kouchekian and the author \cite{KouPro}, in which similar
comparison ideas were used for the diagonals of the Walsh table; see also
\cite{ProRat}.

\medskip
\noindent
A natural question is whether the upper bounds in \eqref{eq:product-upper} and in Theorem~\ref{thm:first-main} are sharp. At present, the proofs of the corresponding equalities require additional geometric assumptions on \(E\).
More precisely, when \(E\) is a continuum with connected complement and Jordan boundary, the corresponding product on \(E_R\) admits an exact asymptotic formula.

\begin{theorem}\label{thm:two}
Assume that \(E\) is a continuum with connected complement and Jordan boundary. Fix \(m\ge0\). Then, for every \(R\) with \(1\le R<R_0\),
\[
\limsup_{n\to\infty}
\left(\prod_{k=0}^{m}\rho_{n-m+k,k}(f;E_R)\right)^{1/n}
=
\frac{R^{m+1}}{R_0R_1\cdots R_m}.
\]
\end{theorem}
\medskip

\noindent
Combining Theorems ~\ref{thm:first-main} and~\ref{thm:two}, we obtain that the estimate in Theorem~\ref{thm:first-main} is sharp.
\begin{theorem}\label{thm:equality-case}
Assume that \(E\) is a continuum with connected complement and Jordan boundary.
 Fix \(m\ge0\), assume that \(R_m<\infty\), and let \(R\) satisfy \(1<R<R_0\). Then
\[
\limsup_{n\to\infty}
\left(
\prod_{k=0}^{m}
\frac{\rho_{n-m+k,k}(f;E)}
{\rho_{n-m+k,k}(f;E_R)}
\right)^{1/n}
=
\frac{1}{R^{m+1}}.
\]
\end{theorem}

\begin{theorem}\label{thm:common-subsequence}
Assume that \(E\) is a continuum with connected complement and Jordan boundary.
Fix \(m\ge 0\), and assume that \(R_m<\infty\). Then there exists a subsequence
\(\Lambda\subset\mathbb N\) such that, for every \(k=0,1,\dots,m\) and every \(R\) with \(1\le R<R_0\), one has
\[
\rho_{n-m+k,k}(f;E_R)^{1/n}\to \frac{R}{R_k},
\qquad n\to\infty,\ \ n\in\Lambda,
\]
and
\[
\left(\prod_{j=0}^{k}\rho_{n-m+j,j}(f;E_R)\right)^{1/n}
\to
\frac{R^{k+1}}{R_0R_1\cdots R_k},
\qquad n\to\infty,\ \ n\in\Lambda.
\]
In addition, for every \(R\) with \(1<R<R_0\),
\[
\left(
\prod_{j=0}^{k}
\frac{\rho_{n-m+j,j}(f;E)}
{\rho_{n-m+j,j}(f;E_R)}
\right)^{1/n}
\to
\frac{1}{R^{k+1}},
\qquad n\to\infty,\ \ n\in\Lambda.
\]
\end{theorem}

\noindent
For each fixed \(m\), the subsequence \(\Lambda\) is chosen so that the product
\[\prod_{j=0}^{m}\rho_{n-m+j,j}(f;E)\]
attains, along \(\Lambda\), the limsup in Theorem~\ref{thm:two} with \(R=1\):
\[
\lim_{n\to\infty, n\in\Lambda}
\left(\prod_{k=0}^{m}\rho_{n-m+k,k}(f;E)\right)^{1/n}
=
\frac{1}{R_0R_1\cdots R_m}.
\]
Thus, along \(\Lambda\), the extremal exponential behavior is attained simultaneously
for all rows \(k=0,1,\dots,m\) and for all partial products up to level \(k\), both on \(E\)
and on the Green sublevel sets \(E_R\), \(1<R<R_0\).
In this sense, a subsequence selected from the product of the approximation errors from the first
\(m+1\) rows of the Walsh table on \(E\) already determines the whole finite hierarchy below it.

\medskip
\noindent
For \(R>0\), let
\[
D_R:=\{z\in\C:\ |z|<R\}.
\]
We write
\[
\overline{D}_R:=\{z\in\C:\ |z|\le R\}.
\]
In particular, the open unit disc is \(\D=D_1\).

\medskip
\noindent
The next theorem treats the disc case \(E=\overline{\D}\), where the product asymptotics admit a classical interpretation in terms of Hankel determinants. This provides a direct connection with Hadamard's theorem and clarifies the analogy underlying our results.

\begin{theorem}\label{thm:hankel-disc-quotient-R}
Fix \(m\ge0\), assume that \(R_m<\infty\), and let \(R\) satisfy \(1<R<R_0\).
Then
\[
\limsup_{n\to\infty}
\left(
\frac{\displaystyle\bigl|\det K_{n+1,m+1}(f)\bigr|}
{\displaystyle\prod_{k=0}^{m}\rho_{\,n-m+k,\,k}\bigl(f; \overline{D}_R \bigr)}
\right)^{1/n}
=
\frac{1}{R^{m+1}},
\]
where \(K_{n+1,m+1}(f)\) denotes the \((m+1)\times (m+1)\) Hankel matrix formed from the Taylor coefficients of \(f\).
\end{theorem}
\medskip

\noindent
\textbf{Organization of the paper.}
Section~\ref{sec:hadamard} recalls Hadamard's classical theorem and the AAK-type background needed for the disc case; see \cite{AAK1,AAK2,Peller}. Section~\ref{sec:proofs-unit-disc} treats the case of the unit disc and establishes the corresponding results in the disc case, including Theorem~\ref{thm:hankel-disc-quotient-R} and the disc analogues of Theorems~\ref{thm:first-main}, \ref{thm:two}, and \ref{thm:equality-case}. In Section~\ref{sec:four} we then pass to continua by means of the exterior conformal map and the corresponding Faber polynomial expansion. The proof of Theorem~\ref{thm:common-subsequence} is given in the final subsection of Section~\ref{sec:four}.

\section{Hadamard's Theorem, Hankel Determinants, and AAK Theory}\label{sec:hadamard}
\subsection{Power series, Hankel determinants, and Hadamard's theorem}

Suppose that \(f\) is analytic at \(0\) and is represented near \(0\) by the convergent series
\begin{equation}\label{eq:series}
f(z)=\sum_{k=0}^{\infty} f_k z^k .
\end{equation}
Denote by \(R_0\) the radius of convergence of the series \eqref{eq:series}, and assume that \(R_0>1\).
For each integer \(m\ge 1\), let \(R_m\) be the supremum of all \(R>0\) such that \(f\)
admits a meromorphic continuation to \(D_R\) with at most \(m\) poles, counted with multiplicity. Then
\[
R_0\le R_1\le R_2\le\cdots .
\]

Fix \(m\ge 1\). For \(n\ge 0\), define the \(m\times m\) Hankel matrix
\[
K_{n,m}:=
\begin{bmatrix}
f_{n+m-1} & f_{n+m-2} & \cdots & f_{n}\\
f_{n+m-2} & f_{n+m-3} & \cdots & f_{n-1}\\
\vdots    & \vdots    & \ddots & \vdots\\
f_{n}     & f_{n-1}   & \cdots & f_{n-m+1}
\end{bmatrix},
\]
where, as usual, \(f_k:=0\) for \(k<0\). Thus \(K_{n,1}=[f_n]\). Set \(l_0:=1\) and, for \(m\ge 1\),
\[
l_m:=\limsup_{n\to\infty} \bigl|\det K_{n,m}\bigr|^{1/n}.
\]

A classical theorem of Hadamard \cite{Ha} (see also \cite{Di}) expresses the meromorphy radii \(\{R_m\}\) in terms of these Hankel-determinant asymptotics.

\medskip
\noindent
\textbf{Hadamard's theorem.}
For \(m=0,1,2,\dots\),
\[
R_m=\frac{l_m}{l_{m+1}}.
\]
Here \(R_m=\infty\) if \(l_1,\dots,l_m\neq 0\) and \(l_{m+1}=0\).

In particular, for every \(m\ge 1\),
\begin{equation}\label{eq:Hadamard_product}
\limsup_{n\to\infty} \bigl|\det K_{n,m}\bigr|^{1/n}
=\frac{1}{R_0R_1\cdots R_{m-1}}.
\end{equation}
Hadamard's theorem shows that Hankel determinants encode the radii of meromorphic continuation and, in particular, yield the product relation \eqref{eq:Hadamard_product}. In this paper, we develop an analogue of \eqref{eq:Hadamard_product} for rational approximation on a continuum with connected complement and Jordan boundary, expressing the asymptotics of products of best rational approximation errors in terms of the corresponding radii of meromorphic continuation.

\subsection{Hankel operators, bilinear forms, and AAK-type estimates on domains with analytic boundary}\label{sec:aak}

Let \(G\subset\C\) be a bounded domain whose boundary \(\Gamma:=\partial G\) consists of \(N\) disjoint analytic Jordan curves, positively oriented with respect to \(G\), and assume that \(0\in G\).

For \(1\le p\le\infty\), let \(E_p(G)\) denote the Smirnov class on \(G\), identified with its boundary values on \(\Gamma\); see \cite{Koo,TumarkinKhavinson1955,TumarkinKhavinson1958}. We shall use the standard characterization that a function \(\varphi\in L_p(\Gamma)\) is the boundary value of a function in \(E_p(G)\) if and only if
\[
\int_\Gamma \frac{\varphi(t)\,dt}{t-z}=0,
\qquad z\in\widehat{\C}\setminus\overline{G}.
\]

Fix an integer \(l\ge 0\). Denote by \(L_{2,l}(\Gamma)\) the Hilbert space endowed with the weighted inner product
\[
\langle u,v\rangle_{2,l}:=\int_\Gamma u(t)\overline{v(t)}\,|t|^l\,|dt|.
\]
Denote by \(H_{2,l}(G)\) the class of functions on \(\Gamma\) consisting of
\[
q(t)=\frac{\varphi(t)}{t^l},
\qquad \varphi\in E_2(G).
\]
We regard \(H_{2,l}(G)\) as a closed subspace of \(L_{2,l}(\Gamma)\), and write
\[
L_{2,l}(\Gamma)=H_{2,l}(G)\oplus H_{2,l}(G)^\perp .
\]
We now record the orthogonality characterization needed below: a function
\(u\in L_{2,l}(\Gamma)\) belongs to \(H_{2,l}(G)^\perp\) if and only if there exists
\(v\in E_2(G)\) such that
\[
\overline{u(t)}\,|t|^l\,|dt|=v(t)t^l\,dt
\]
for almost every \(t\in\Gamma\).

Let \(f\in C(\Gamma)\). Associated with the symbol \(f\), we consider the weighted Hankel operator
\[
A_{f,G}:E_2(G)\to H_{2,l}(G)^\perp,
\qquad
A_{f,G}\alpha:=P_-(f\alpha),
\]
where \(P_-\) is the orthogonal projection from \(L_{2,l}(\Gamma)\) onto
\(H_{2,l}(G)^\perp\); see \cite{PS2000}. The operator \(A_{f,G}\) is compact.

Motivated by the Hardy--Smirnov approach of \cite{PPu1}, we consider the bilinear Hankel form
\[
[u,v]_{f,G}:=\int_\Gamma u(t)v(t)f(t)\,t^l\,dt,
\qquad u,v\in E_2(G).
\]
This form generates an antilinear compact operator
\[
B_{f,G}:E_2(G)\to E_2(G)
\]
defined by
\[
[u,v]_{f,G}=\langle u,B_{f,G}v\rangle_{2,l}
=\langle v,B_{f,G}u\rangle_{2,l},
\qquad u,v\in E_2(G).
\]

If \(\alpha\in E_2(G)\), then there exists \(\beta\in E_2(G)\) such that
\begin{equation}\label{eq:weighted_fu_beta}
\bigl(\alpha(t)f(t)-\beta(t)/t^l\bigr)t^l\,dt
=
\overline{(B_{f,G}\alpha)(t)}\,|t|^l\,|dt|
\end{equation}
for almost every \(t\in\Gamma\).

We  write
\[
s_k(A_{f,G}),\qquad k=0,1,2,\dots,
\]
for the singular numbers of \(A_{f,G}\), arranged in nonincreasing order and
repeated according to multiplicity; see \cite{GKr}.

For \(n\ge 0\), let \(\mathcal M_{n+l,n}(G)\) denote the class of functions representable in the form
\[
h=\frac{p}{q\,t^l},
\]
where \(p\in E_\infty(G)\), \(q\) is a polynomial of degree at most \(n\), and \(q\not\equiv0\). Set
\[
\Delta_{n+l,n}(f;\Gamma):=\inf_{h\in\mathcal M_{n+l,n}(G)}\|f-h\|_\infty .
\]
Then the weighted Adamyan--Arov--Kre\u{\i}n theorem for multiply connected domains yields \cite{PS2000,Pro1}
\[
\Delta_{n+N-1+l,n+N-1}(f;\Gamma)\le s_n(A_{f,G})\le \Delta_{n+l,n}(f;\Gamma),
\qquad n\ge N-1.
\]
In the simply connected case, this estimate becomes an equality:
\begin{equation}\label{eq:AAK_general}
\Delta_{n+l,n}(f;\Gamma)=s_n(A_{f,G}),
\qquad n=0,1,2,\dots .
\end{equation}
In particular, for the classes \(\mathcal M_{n+l,n}(\D)\) on the unit disc \(\D\),
the same equality also follows by applying the classical Adamyan--Arov--Kre\u{\i}n theorem
to \(z^l f\); see \cite{AAK1,AAK2}. See also \cite{KPS1998,PS1999}
for related weighted and approximation-theoretic applications.

We have
\[
A_{f,G}^\ast A_{f,G}=B_{f,G}^2.
\]
Indeed, for \(u,v\in E_2(G)\),
\[
\langle A_{f,G}u,A_{f,G}v\rangle_{2,l}
=\int_\Gamma (A_{f,G}u)(t)\overline{A_{f,G}v(t)}\,|t|^l\,|dt|
=\langle B_{f,G}v,B_{f,G}u\rangle_{2,l}
=\langle u,B_{f,G}^2v\rangle_{2,l}.
\]

Let \(\{s_k(A_{f,G})\}_{k=0}^\infty\) be the singular numbers of \(A_{f,G}\), that is, the eigenvalues of \((B_{f,G}^2)^{1/2}\), and let \(\{\alpha_k\}_{k=0}^\infty\) be a corresponding orthonormal system of eigenfunctions of \(B_{f,G}\):
\[
B_{f,G}\alpha_k=s_k(A_{f,G})\,\alpha_k,
\qquad k=0,1,2,\dots .
\]
Then
\begin{equation}\label{eq:double_orthogonality}
\langle \alpha_i,\alpha_j\rangle_{2,l}=\delta_{ij},
\qquad
[\alpha_i,\alpha_j]_{f,G}=s_j(A_{f,G})\,\delta_{ij},
\qquad i,j=0,1,2,\dots .
\end{equation}
By \eqref{eq:weighted_fu_beta}, there exist functions \(\beta_k\in E_2(G)\) such that
\[
\bigl(\alpha_k(t)f(t)-\beta_k(t)/t^l\bigr)t^l\,dt
=
s_k(A_{f,G})\,\overline{\alpha_k(t)}\,|t|^l\,|dt|,
\qquad k=0,1,2,\dots,
\]
for almost every \(t\in\Gamma\).

It follows from \eqref{eq:double_orthogonality} that
\[
\int_\Gamma \alpha_i(t)\alpha_j(t)f(t)\,t^l\,dt
=
s_j(A_{f,G})\,\delta_{ij},
\qquad i,j=0,1,2,\dots ,
\]
and hence the product of the first \(m+1\) singular numbers admits the determinantal representation
\begin{equation}\label{eq:det_identity}
\prod_{k=0}^{m}s_k(A_{f,G})
=
\det\left(
\int_\Gamma \alpha_i(t)\alpha_j(t)f(t)\,t^l\,dt
\right)_{i,j=0}^{m},
\qquad m=0,1,2,\dots .
\end{equation}

We shall also use the following inequality; see \cite{PPu1}. For any \(\varphi_0,\dots,\varphi_m\in E_2(G)\),
\begin{equation}\label{eq:PP_det}
\left|
\det\left(
\int_\Gamma \varphi_i(t)\varphi_j(t)f(t)\,t^l\,dt
\right)_{i,j=0}^{m}
\right|
\le
\Bigl(\prod_{k=0}^{m}s_k(A_{f,G})\Bigr)
\det\Bigl(\langle \varphi_i,\varphi_j\rangle_{2,l}\Bigr)_{i,j=0}^{m}.
\end{equation}

In what follows, the weighted AAK identity \eqref{eq:AAK_general}, the determinantal identity \eqref{eq:det_identity}, and inequality \eqref{eq:PP_det} will be the basic tools for passing from singular numbers of Hankel operators to asymptotic estimates for products of rational approximation errors. 
\section{The Disc Case}\label{sec:proofs-unit-disc}
In this section we treat the case $E=\overline{\D}$. We first use the inversion $z\mapsto 1/z$ to rewrite rational approximation on $\overline{\D}$ as a corresponding meromorphic approximation problem for the transformed function $F(z)=f(1/z)$. We then use the weighted AAK theorem in the disc together with Hadamard's theorem on Hankel determinants to prove 
the disc-case analogues of Theorems~\ref{thm:first-main}, \ref{thm:two}, and \ref{thm:equality-case}, and Theorem~\ref{thm:hankel-disc-quotient-R}.

\subsection{Passage to the exterior problem and the disc AAK setup}\label{subsec:unit_disc_theorems}

Assume that $f$ is analytic at $0$ and has a Taylor expansion
\[
f(z)=\sum_{k=0}^{\infty} f_k z^k,
\qquad R_0>1,
\]
where $R_0$ is the radius of convergence.

First we consider the disc-case form of Theorem~\ref{thm:two} for $R=1$.

\begin{proposition}\label{prop:disc_product}
Fix $m\ge0$. Then
\[
\limsup_{n\to\infty}
\left(\prod_{k=0}^{m}\rho_{n-m+k,k}(f;\overline{\D})\right)^{1/n}
=\frac{1}{R_0R_1\cdots R_m}.
\]
\end{proposition}
To prove this proposition, we reduce the asymptotic behavior of the product
\[
\prod_{k=0}^{m}\rho_{n-m+k,k}(f;\overline{\D})
\]
to singular-number and Hankel-determinant asymptotics, and then apply Hadamard's theorem.

Consider the inversion $z\mapsto 1/z$ and define
\[
F(z):=f\!\left(\frac{1}{z}\right).
\]
Then $F$ is analytic in a neighborhood of $\infty$ and has the Laurent expansion
\[
F(z)=\sum_{k=0}^{\infty} f_k z^{-k}.
\]
Since adding or subtracting a constant does not affect best approximation errors, we may assume without loss of generality that $f(0)=0$, and hence $F(\infty)=0$.

For integers $N,M\ge0$, recall that
\[
\mathcal R_{N,M}
=
\left\{
r:\ r(z)=\frac{p(z)}{q(z)},\ 
\deg p\le N,\ \deg q\le M,\ q\not\equiv0
\right\}.
\]
We also introduce the class
\[
\mathcal R^*_{N,M}
=
\left\{
r:\ r(z)=\frac{p(z)}{z^{\,N-M}q(z)},\ 
\deg p\le N,\ \deg q\le M,\ q\not\equiv0
\right\},
\qquad N,M\ge0.
\]
For each \(k\) with \(0\le k\le m\), the inversion \(z\mapsto 1/z\) maps \(\mathcal R_{n-m+k,k}\) onto \(\mathcal R^*_{n-m+k,k}\).

For $R>0$, set
\[
K_R:=\{z\in\C:\ |z|\ge R\}.
\]
In particular,
\[
K_1=\C\setminus\D.
\]
Since \(m\) is fixed and \(n\to\infty\), we may ignore finitely many initial terms and therefore assume that \(n\ge m\) in what follows. Recalling that $F(z)=f(1/z)$ on $K_1$, we have, for $0\le k\le m$,
\[
\rho_{n-m+k,k}(f;\overline{\D})
=
\rho^*_{n-m+k,k}(F;K_1),
\]
where
\[
\rho^*_{N,M}(F;K_1)
:=
\inf_{r\in\mathcal R^*_{N,M}}\|F-r\|_{K_1},
\qquad
\|g\|_{K_1}:=\sup_{z\in K_1}|g(z)|.
\]
Consequently,
\[
\prod_{k=0}^{m}\rho_{n-m+k,k}(f;\overline{\D})
=
\prod_{k=0}^{m}\rho^*_{n-m+k,k}(F;K_1).
\]
Therefore, it suffices to prove that
\[
\limsup_{n\to\infty}
\left(
\prod_{k=0}^{m}\rho^*_{n-m+k,k}(F;K_1)
\right)^{1/n}
=
\frac{1}{R_0R_1\cdots R_m}.
\]
We now specialize the Hankel-operator setup of Subsection~\ref{sec:aak} to the case \(G=\D\), \(\Gamma=\T=\{z\in\C:\ |z|=1\}\). Let \(L_{2,l}(\T)\) be the weighted space on \(\T\) with inner product
\[
\langle \varphi,\psi\rangle_{2,l}   
:=
\int_{\T}\varphi(t)\overline{\psi(t)}\,|t|^l\,|dt|
=
\int_{\T}\varphi(t)\overline{\psi(t)}\,|dt|.
\]
Let \(E_2(\D)\) be the Smirnov class on \(\D\), viewed as a subspace of \(L_{2,l}(\T)\). In the unit disc, \(E_2(\D)\) coincides with the classical Hardy space \(H^2(\D)\), although we retain the weighted notation \(L_{2,l}(\T)\) and \(\langle\cdot,\cdot\rangle_{2,l}\) in order to keep the connection with the general setup visible.

Since $F$ is continuous on \(\T\), we consider the corresponding Hankel operator $A_{F,\D}$ in the weighted $L_2$-setting, and write
\[
s_k(A_{F,\D}),
\qquad k=0,1,2,\dots,
\]
for its singular numbers, arranged in nonincreasing order and repeated according to multiplicity.

Recall that, for each \(k\ge0\), \(\mathcal M_{k+l,k}(\D)\) denotes the class of functions
representable on \(\T\) in the form
\[
h=\frac{p}{q\,t^l},
\qquad p\in E_\infty(\D),\ \deg q\le k,\ q\not\equiv0,
\]
and that
\[
\Delta_{k+l,k}(F;\T):=\inf_{h\in\mathcal M_{k+l,k}(\D)}\|F-h\|_\infty.
\]
Then
\[
s_k(A_{F,\D})=\Delta_{k+l,k}(F;\T), \qquad k\ge 0,
\]
by the weighted AAK theorem.

\subsection{The determinantal upper bound and the proof of Proposition~\ref{prop:disc_product}}\label{subsec:disc_det_product}

For integers $m\ge0$ and $n\ge m$, consider the $(m+1)\times(m+1)$ Hankel matrix
\[
K_{n+1,m+1}(f):=\bigl(f_{n+m+1-i-j}\bigr)_{i,j=0}^{m}.
\]
Using the Laurent expansion of \(F(z)=f(1/z)\) and Cauchy's formula on the positively oriented circle \(\T\), we have
\[
f_k=\frac{1}{2\pi i}\int_{\T}F(t)\,t^{k-1}\,dt.
\]
Hence
\begin{equation}\label{eq:Hankel_as_form_det_disc}
\det K_{n+1,m+1}(f)
=
\left(\frac{1}{2\pi i}\right)^{m+1}
\det\left(
\int_{\T}(t^{m-i}t^{m-j})\,F(t)\,t^{n-m}\,dt
\right)_{i,j=0}^{m}.
\end{equation}
Applying the determinantal inequality \eqref{eq:PP_det} with \(l=n-m\) and the  functions
\[
\varphi_i(t)=t^{m-i},
\qquad i=0,1,\dots,m,
\]
we obtain
\begin{equation}\label{eq:Hankel_upper_via_singular_disc}
\left|
\det\left(
\int_{\T}(t^{m-i}t^{m-j})\,F(t)\,t^{n-m}\,dt
\right)_{i,j=0}^{m}
\right|
\le
\Bigl(\prod_{k=0}^{m}s_k(A_{F,\D})\Bigr)\,
\det\Bigl(\langle t^{m-i},t^{m-j}\rangle_{2,n-m}\Bigr)_{i,j=0}^{m}.
\end{equation}
Since $|t|=1$ on $\T$, the weighted inner product \(\langle\cdot,\cdot\rangle_{2,n-m}\) coincides with the usual $L_2(\T)$ inner product, and the Gram matrix is diagonal. Therefore,
\[
\det\Bigl(\langle t^{m-i},t^{m-j}\rangle_{2,n-m}\Bigr)_{i,j=0}^{m}
=(2\pi)^{m+1}.
\]
Combining this with \eqref{eq:Hankel_as_form_det_disc} and \eqref{eq:Hankel_upper_via_singular_disc}, we obtain
\[
|\det K_{n+1,m+1}(f)|
\le
\prod_{k=0}^{m}s_k(A_{F,\D}).
\]

By the weighted AAK identity in the disc,
\[
s_k(A_{F,\D})=\Delta_{n-m+k,k}(F;\T),
\qquad k=0,1,\dots,m,
\]
and hence
\[
|\det K_{n+1,m+1}(f)|
\le
\prod_{k=0}^{m}\Delta_{n-m+k,k}(F;\T).
\]

Moreover, for $0\le k\le m$,
\[
\Delta_{n-m+k,k}(F;\T)\le \rho^{*}_{n-m+k,k}(F;K_1),
\]
since $\mathcal R^{*}_{n-m+k,k}\subset \mathcal M_{n-m+k,k}(\D)$ and $\T\subset K_1$. Using
\[
\rho_{n-m+k,k}(f;\overline{\D})
=
\rho^{*}_{n-m+k,k}(F;K_1),
\qquad 0\le k\le m,
\]
we obtain
\begin{equation}\label{eq:Hankel_det_vs_rho_disc}
|\det K_{n+1,m+1}(f)|
\le
\prod_{k=0}^{m}\rho_{n-m+k,k}(f;\overline{\D}).
\end{equation}

Hadamard's theorem gives
\[
\limsup_{n\to\infty}|\det K_{n+1,m+1}(f)|^{1/n}
=
\frac{1}{R_0R_1\cdots R_m}.
\]
Moreover, by the Saff--Gonchar theorem, for each fixed \(m\),
\[
\limsup_{n\to\infty}
\left(
\prod_{k=0}^{m}\rho_{n-m+k,k}(f;\overline{\D})
\right)^{1/n}
\le
\frac{1}{R_0R_1\cdots R_m}.
\]
Together with the inequality \eqref{eq:Hankel_det_vs_rho_disc}, this yields
\[
\limsup_{n\to\infty}
\left(
\prod_{k=0}^{m}\rho_{n-m+k,k}(f;\overline{\D})
\right)^{1/n}
=
\frac{1}{R_0R_1\cdots R_m}.
\]
This proves Proposition~\ref{prop:disc_product}.
\subsection{Proof of Theorem~\ref{thm:two} in the disc case}\label{subsec:disc_scaled_product}

We now pass from Proposition~\ref{prop:disc_product}, which treats the case $R=1$, to arbitrary $R$ with $1\le R<R_0$.

Fix $R$ with $1\le R<R_0$, and define
\[
g(z):=f(Rz).
\]
Then $g$ is analytic at $0$. For each $k\ge0$, let \(R_k(g)\) denote the radius of meromorphic continuation of \(g\) with at most \(k\) poles. Then
\[
R_k(g)=\frac{R_k}{R}.
\]
Indeed, $g$ admits a meromorphic continuation to $|z|<\rho$ with at most $k$ poles if and only if $f$ admits a meromorphic continuation to $|z|<R\rho$ with at most $k$ poles.

Next, for every $n,m\ge0$,
\[
\rho_{n,m}(g;\overline{\D})=\rho_{n,m}(f;\overline{D}_R).
\]
This follows by the change of variables \(w=Rz\). Indeed, if \(r\in\mathcal R_{n,m}\), then
\[
\|g-r\|_{\overline{\D}}
=
\sup_{|z|\le1}|f(Rz)-r(z)|
=
\sup_{|w|\le R}\left|f(w)-r\!\left(\frac{w}{R}\right)\right|.
\]
Since \(r(w/R)\in\mathcal R_{n,m}\), taking the infimum over \(r\in\mathcal R_{n,m}\) gives the desired relation.

Therefore, applying Proposition~\ref{prop:disc_product} to the function $g$, we obtain
\[
\limsup_{n\to\infty}
\left(\prod_{k=0}^{m}\rho_{n-m+k,k}(g;\overline{\D})\right)^{1/n}
=
\frac{1}{R_0(g)R_1(g)\cdots R_m(g)}.
\]
Using the two identities above, this becomes
\[
\limsup_{n\to\infty}
\left(\prod_{k=0}^{m}\rho_{n-m+k,k}(f;\overline{D}_R)\right)^{1/n}
=
\frac{1}{(R_0/R)(R_1/R)\cdots (R_m/R)}
=
\frac{R^{m+1}}{R_0R_1\cdots R_m}.
\]
Thus, in the disc case,
\[
\limsup_{n\to\infty}
\left(\prod_{k=0}^{m}\rho_{n-m+k,k}(f;\overline{D}_R)\right)^{1/n}
=
\frac{R^{m+1}}{R_0R_1\cdots R_m},
\qquad 1\le R<R_0,
\]
which proves the disc-case form of Theorem~\ref{thm:two}.


\subsection{Proof of Theorem~\ref{thm:hankel-disc-quotient-R}}

We begin with  $R=1$, that is, $\overline{D}_1=\overline{\D}$. By \eqref{eq:Hankel_det_vs_rho_disc},
\[
|\det K_{n+1,m+1}(f)|
\le
\prod_{k=0}^{m}\rho_{n-m+k,k}(f;\overline{\D}).
\]
Dividing by \(\prod_{k=0}^{m}\rho_{n-m+k,k}(f;\overline{\D})\), taking \(n\)th roots, and passing to the limit superior, we obtain
\[
\limsup_{n\to\infty}
\left(
\frac{|\det K_{n+1,m+1}(f)|}
{\displaystyle\prod_{k=0}^{m}\rho_{n-m+k,k}(f;\overline{\D})}
\right)^{1/n}
\le 1.
\]

Hadamard's theorem gives
\begin{equation}\label{eq:hadamard_repeat_disc}
\limsup_{n\to\infty}|\det K_{n+1,m+1}(f)|^{1/n}
=
\frac{1}{R_0R_1\cdots R_m},
\end{equation}
while Proposition~\ref{prop:disc_product} yields
\begin{equation}\label{eq:disc_product_repeat_R1}
\limsup_{n\to\infty}
\left(
\prod_{k=0}^{m}\rho_{n-m+k,k}(f;\overline{\D})
\right)^{1/n}
=
\frac{1}{R_0R_1\cdots R_m}.
\end{equation}
By the elementary inequality
\[
\limsup_{n\to\infty}\frac{a_n}{b_n}
\ge
\frac{\limsup_{n\to\infty}a_n}{\limsup_{n\to\infty}b_n},
\qquad a_n>0,\ b_n>0,
\]
together with \eqref{eq:hadamard_repeat_disc} and \eqref{eq:disc_product_repeat_R1}, we obtain the reverse inequality
\[
\limsup_{n\to\infty}
\left(
\frac{|\det K_{n+1,m+1}(f)|}
{\displaystyle\prod_{k=0}^{m}\rho_{n-m+k,k}(f;\overline{\D})}
\right)^{1/n}
\ge 1.
\]
Hence
\begin{equation}\label{eq:disc_hankel_quotient_R1}
\limsup_{n\to\infty}
\left(
\frac{|\det K_{n+1,m+1}(f)|}
{\displaystyle\prod_{k=0}^{m}\rho_{n-m+k,k}(f;\overline{\D})}
\right)^{1/n}
=1.
\end{equation}

We now pass to general $R$, $1<R<R_0$, and 
define
\[
g(z):=f(Rz).
\]
Then, for each $k\ge0$,
\[
R_k(g)=\frac{R_k}{R},
\]
and, as in the proof of Proposition~\ref{prop:disc_product},
\[
\rho_{n,m}(g;\overline{\D})=\rho_{n,m}(f;\overline{D}_R).
\]
Moreover, the Taylor coefficients of $g$ are $g_j=f_jR^j$, and therefore
\[
K_{n+1,m+1}(g)
=
\bigl(f_{n+m+1-i-j}R^{\,n+m+1-i-j}\bigr)_{i,j=0}^{m}.
\]
Factoring $R^{n+m+1-i}$ from the $i$th row and $R^{-j}$ from the $j$th column, we obtain
\[
\det K_{n+1,m+1}(g)=R^{(m+1)(n+1)}\det K_{n+1,m+1}(f).
\]
Hence
\begin{equation}\label{eq:disc_hankel_scaling_identity}
\frac{|\det K_{n+1,m+1}(g)|}
{\displaystyle\prod_{k=0}^{m}\rho_{n-m+k,k}(g;\overline{\D})}
=
\frac{R^{(m+1)(n+1)}|\det K_{n+1,m+1}(f)|}
{\displaystyle\prod_{k=0}^{m}\rho_{n-m+k,k}(f;\overline{D}_R)}.
\end{equation}
Applying \eqref{eq:disc_hankel_quotient_R1} to $g$, we obtain
\[
\limsup_{n\to\infty}
\left(
\frac{|\det K_{n+1,m+1}(g)|}
{\displaystyle\prod_{k=0}^{m}\rho_{n-m+k,k}(g;\overline{\D})}
\right)^{1/n}
=1.
\]
Using \eqref{eq:disc_hankel_scaling_identity}, this yields
\[
\limsup_{n\to\infty}
\left(
\frac{|\det K_{n+1,m+1}(f)|}
{\displaystyle\prod_{k=0}^{m}\rho_{n-m+k,k}(f;\overline{D}_R)}
\right)^{1/n}
=
\frac{1}{R^{m+1}}.
\]
This proves Theorem~\ref{thm:hankel-disc-quotient-R}.


\subsection{Comparison of products on two circles}\label{subsec:two-circle-comparison}

We shall need two comparison results: one for meromorphic approximation on two concentric circles, and one between meromorphic and rational approximation errors.
The argument is similar in spirit to the comparison method used in \cite{KouPro}.

For \(r>0\), set
\[
T_r:=\{z\in\C:\ |z|=r\}.
\]
Let
\[
\frac{1}{R_0}<\rho<r<1.
\]
Fix \(l=n-m\). We consider the weighted Hankel operators on the discs \(D_r\) and \(D_\rho\), with boundaries \(T_r\) and \(T_\rho\), associated with \(F\), as in Section~\ref{sec:hadamard}.

By the weighted AAK theorem, the corresponding singular numbers
coincide with the meromorphic approximation errors on these circles. More precisely,
\[
s_k(A_{F,D_r})
=
\Delta_{n-m+k,k}(F;T_r),
\qquad
s_k(A_{F,D_\rho})
=
\Delta_{n-m+k,k}(F;T_\rho),
\]
for \(k=0,1,\dots,m\).

We now compare the corresponding products.

\begin{lemma}\label{lem:circle-comparison}
Let
\[
\frac1{R_0}<\rho<r<1.
\]
Then, for every fixed integer \(m\ge0\), there exists a constant \(c(r,\rho)>0\), independent of \(n\), such that
\[
\prod_{k=0}^{m}\Delta_{n-m+k,k}(F;T_r)
\le
c(r,\rho)\left(\frac{\rho}{r}\right)^{n(m+1)}
\prod_{k=0}^{m}\Delta_{n-m+k,k}(F;T_\rho).
\]
\end{lemma}

\begin{proof}
For the discs \(D_r\) and \(D_\rho\), with boundaries \(T_r\) and \(T_\rho\), respectively, we
consider the weighted Hankel operators corresponding to the parameter \(l=n-m\).
Applying \eqref{eq:det_identity} with \(G=D_r\), symbol \(F\), and \(l=n-m\), we obtain
\[
\prod_{k=0}^{m}s_k\bigl(A_{F,D_r}\bigr)
=
\det\!\left(
\int_{T_r}\alpha_i(t)\alpha_j(t)F(t)t^{\,n-m}\,dt
\right)_{i,j=0}^{m},
\]
where \(\{\alpha_i\}_{i=0}^{m}\subset E_2(D_r)\) is the orthonormal system corresponding to the first
\(m+1\) singular numbers of \(A_{F,D_r}\).

\medskip
\noindent
We now transfer this determinant from \(T_r\) to \(T_\rho\). Since \(\alpha_i,\alpha_j\in E_2(D_r)\) are holomorphic in \(D_r\), and \(F\) is holomorphic in
\(|z|>1/R_0\), the integrand is holomorphic in the annulus
\[
\{z\in\C:\rho\le |z|\le r\}.
\]
Hence Cauchy's theorem gives
\[
\int_{T_r}\alpha_i(t)\alpha_j(t)F(t)t^{\,n-m}\,dt
=
\int_{T_\rho}\alpha_i(t)\alpha_j(t)F(t)t^{\,n-m}\,dt.
\]
Next, we apply the determinantal inequality \eqref{eq:PP_det} on \(D_\rho\). This gives
\begin{equation}\label{eq:circle_compare_singular}
\prod_{k=0}^{m}s_k\bigl(A_{F,D_r}\bigr)
\le
\prod_{k=0}^{m}s_k\bigl(A_{F,D_\rho}\bigr)
\det\!\left(\langle \alpha_i,\alpha_j\rangle_{2,n-m}\right)_{i,j=0}^{m},
\end{equation}
where \(\langle\cdot,\cdot\rangle_{2,n-m}\) denotes the weighted \(L_2\) inner product on \(T_\rho\) corresponding
to the parameter \(l=n-m\).

We now estimate the Gram determinant on the right. Since \(|t|=\rho\) on \(T_\rho\), each entry of the
Gram matrix has the form
\[
\langle \alpha_i,\alpha_j\rangle_{2,n-m}
=
\int_{T_\rho}\alpha_i(t)\overline{\alpha_j(t)}\,|t|^{\,n-m}\,|dt|
=
\rho^{\,n-m}\int_{T_\rho}\alpha_i(t)\overline{\alpha_j(t)}\,|dt|.
\]
Therefore, the common factor \(\rho^{\,n-m}\) may be factored out of each row of the Gram matrix, and
hence
\[
\det\!\left(\langle \alpha_i,\alpha_j\rangle_{2,n-m}\right)_{i,j=0}^{m}
=
\rho^{(n-m)(m+1)}
\det\!\left(
\int_{T_\rho}\alpha_i(t)\overline{\alpha_j(t)}\,|dt|
\right)_{i,j=0}^{m}.
\]

To estimate the remaining determinant, we use the fact that the functions
\(\alpha_0,\dots,\alpha_m\) are orthonormal in \(E_2(D_r)\) with respect to the weighted inner
product corresponding to \(l=n-m\). Since \(|t|=r\) on \(T_r\), this orthonormality means that
\[
\int_{T_r}\alpha_i(t)\overline{\alpha_j(t)}\,|t|^{\,n-m}\,|dt|
=
\delta_{ij},
\qquad i,j=0,1,\dots,m,
\]
and therefore
\[
r^{\,n-m}\int_{T_r}\alpha_i(t)\overline{\alpha_j(t)}\,|dt|
=
\delta_{ij}.
\]
In particular,
\[
\int_{T_r}|\alpha_i(t)|^2\,|dt|=r^{-(n-m)},
\qquad i=0,1,\dots,m.
\]
Since \(\alpha_i\) is holomorphic in \(D_r\), its integral means are nondecreasing with the radius. Hence
\[
\int_{T_\rho}|\alpha_i(t)|^2\,|dt|
\le
\int_{T_r}|\alpha_i(t)|^2\,|dt|
=r^{-(n-m)},
\qquad i=0,1,\dots,m.
\]
By the Cauchy--Schwarz inequality, each entry of the matrix
\[
\left(
\int_{T_\rho}\alpha_i(t)\overline{\alpha_j(t)}\,|dt|
\right)_{i,j=0}^{m}
\]
is bounded in absolute value by \(r^{-(n-m)}\). Since \(m\) is fixed, its determinant is bounded by
\[
c_0(r,\rho)\,r^{-(n-m)(m+1)}.
\]
Substituting this into the preceding identity, we obtain
\[
\det\!\left(\langle \alpha_i,\alpha_j\rangle_{2,n-m}\right)_{i,j=0}^{m}
\le
c_0(r,\rho)\left(\frac{\rho}{r}\right)^{(n-m)(m+1)}.
\]
Since
\[
(n-m)(m+1)=n(m+1)-m(m+1),
\]
the fixed factor \(\left(\frac{\rho}{r}\right)^{-m(m+1)}\) may be absorbed into the constant. Thus
\[
\det\!\left(\langle \alpha_i,\alpha_j\rangle_{2,n-m}\right)_{i,j=0}^{m}
\le
c(r,\rho)\left(\frac{\rho}{r}\right)^{n(m+1)}.
\]
Combining this with \eqref{eq:circle_compare_singular}, we get
\[
\prod_{k=0}^{m}s_k\bigl(A_{F,D_r}\bigr)
\le
c(r,\rho)\left(\frac{\rho}{r}\right)^{n(m+1)}
\prod_{k=0}^{m}s_k\bigl(A_{F,D_\rho}\bigr).
\]

Finally, by the weighted AAK identity,
\[
s_k(A_{F,D_r})=\Delta_{n-m+k,k}(F;T_r),
\qquad
s_k(A_{F,D_\rho})=\Delta_{n-m+k,k}(F;T_\rho),
\qquad k=0,1,\dots,m.
\]
Therefore
\[
\prod_{k=0}^{m}\Delta_{n-m+k,k}(F;T_r)
\le
c(r,\rho)\left(\frac{\rho}{r}\right)^{n(m+1)}
\prod_{k=0}^{m}\Delta_{n-m+k,k}(F;T_\rho),
\]
which proves the lemma.
\end{proof}

We shall also need a comparison between meromorphic and rational approximation errors. Recall that
\[
K_r:=\{z\in\C:\ |z|\ge r\}.
\]

\begin{lemma}\label{lem:rho-star-comparison}
Let
\[
\frac1{R_0}<\rho<r<1.
\]
Then, for every fixed integer $m\ge 0$, there exists a constant $C(r,\rho)>0$,
independent of $n$, such that
\[
\prod_{k=0}^{m}\rho^*_{n-m+k,k}(F;K_1)
\le
C(r,\rho)\left(\frac{\rho}{r}\right)^{n(m+1)}
\prod_{k=0}^{m}\rho^*_{n-m+k,k}(F;K_\rho).
\]
\end{lemma}
\begin{proof}
Fix \(m\ge 0\) and \(k\in\{0,1,\dots,m\}\).
Let \(h\in \mathcal M_{n-m+k,k}(D_r)\) be arbitrary. For \(z\in K_1\), define
\[
R(z):=\frac{1}{2\pi i}\int_{T_r}\frac{h(t)\,dt}{z-t}.
\]
Then \(R\in \mathcal R^*_{n-m+k,k}\), since \(h\) is the sum of a function holomorphic in \(D_r\)
and a rational function in the class \(\mathcal R^*_{n-m+k,k}\).

Since \(F\) is analytic in the exterior of \(T_r\) and vanishes at infinity, the Cauchy formula for the exterior gives
\[
F(z)-R(z)
=
\frac{1}{2\pi i}\int_{T_r}\frac{F(t)-h(t)}{z-t}\,dt,
\qquad z\in K_1.
\]
Hence
\[
\|F-R\|_{K_1}
\le
\frac{r}{1-r}\,\|F-h\|_\infty,
\]
where \(\|\cdot\|_\infty\) denotes the \(L_\infty\)-norm on \(T_r\).
Since \(R\in\mathcal R^*_{n-m+k,k}\), it follows that
\[
\rho^*_{n-m+k,k}(F;K_1)
\le
\frac{r}{1-r}\,\|F-h\|_\infty.
\]
Since \(h\in \mathcal M_{n-m+k,k}(D_r)\) was arbitrary, we may take the infimum over all such \(h\) and obtain
\[
\rho^*_{n-m+k,k}(F;K_1)
\le
\frac{r}{1-r}\,\Delta_{n-m+k,k}(F;T_r).
\]
Since this estimate holds for each \(k=0,1,\dots,m\), multiplying over \(k\) we obtain
\[
\prod_{k=0}^{m}\rho^*_{n-m+k,k}(F;K_1)
\le
\left(\frac{r}{1-r}\right)^{m+1}
\prod_{k=0}^{m}\Delta_{n-m+k,k}(F;T_r).
\]
Combining this with Lemma~\ref{lem:circle-comparison}, we obtain
\[
\prod_{k=0}^{m}\rho^*_{n-m+k,k}(F;K_1)
\le
C(r,\rho)\left(\frac{\rho}{r}\right)^{n(m+1)}
\prod_{k=0}^{m}\Delta_{n-m+k,k}(F;T_\rho).
\]
Moreover, for every $k=0,1,\dots,m$, we have
\[
\Delta_{n-m+k,k}(F;T_\rho)
\le
\rho^*_{n-m+k,k}(F;K_\rho),
\]
since $\mathcal R^*_{n-m+k,k}\subset \mathcal M_{n-m+k,k}(D_\rho)$ and $T_\rho\subset K_\rho$.

Therefore,
\[
\prod_{k=0}^{m}\rho^*_{n-m+k,k}(F;K_1)
\le
C(r,\rho)\left(\frac{\rho}{r}\right)^{n(m+1)}
\prod_{k=0}^{m}\rho^*_{n-m+k,k}(F;K_\rho),
\]
which proves the lemma.
\end{proof}


\subsection{Proof of Theorem~\ref{thm:first-main} in the disc case}
\label{subsec:first_main_circle}
We now consider the special case \(E=\overline{\D}\). Let
\[
F(z):=f(1/z).
\]
Fix \(R\) with
\[
1<R<R_0,
\]
and choose \(r\) so that
\[
\frac{1}{R}<r<1.
\]
Applying Lemma~\ref{lem:rho-star-comparison} with \(\rho=1/R\), we obtain
\[
\prod_{k=0}^{m}\rho^*_{n-m+k,k}(F;K_1)
\le
C(r,R)\left(\frac{1}{Rr}\right)^{n(m+1)}
\prod_{k=0}^{m}\rho^*_{n-m+k,k}(F;K_{1/R}),
\]
where \(C(r,R)\) is independent of \(n\).

Taking \(n\)th roots and passing to the limit superior, we get
\[
\limsup_{n\to\infty}
\left(
\prod_{k=0}^{m}
\frac{\rho^*_{n-m+k,k}(F;K_1)}
{\rho^*_{n-m+k,k}(F;K_{1/R})}
\right)^{1/n}
\le
\left(\frac{1}{Rr}\right)^{m+1}.
\]
Since \(r\in(1/R,1)\) is arbitrary, letting \(r\to1^{-}\) yields
\[
\limsup_{n\to\infty}
\left(
\prod_{k=0}^{m}
\frac{\rho^*_{n-m+k,k}(F;K_1)}
{\rho^*_{n-m+k,k}(F;K_{1/R})}
\right)^{1/n}
\le
\frac{1}{R^{m+1}}.
\]
Finally, by inversion, for every \(k=0,1,\dots,m\),
\[
\rho_{n-m+k,k}(f;\overline{\D})=\rho^*_{n-m+k,k}(F;K_1),
\qquad
\rho_{n-m+k,k}(f;\overline{D}_R)=\rho^*_{n-m+k,k}(F;K_{1/R}),
\]
and therefore
\[
\limsup_{n\to\infty}
\left(
\prod_{k=0}^{m}
\frac{\rho_{n-m+k,k}(f;\overline{\D})}
{\rho_{n-m+k,k}(f;\overline{D}_R)}
\right)^{1/n}
\le
\frac{1}{R^{m+1}}.
\]
This proves Theorem~\ref{thm:first-main} in the disc case.






\section{From the disc to continua with analytic boundary and the Jordan-boundary case}
\label{sec:four}

In this section we explain how Theorem~\ref{thm:first-main} extends to the case where
\(E\) is a continuum with at least two points and connected complement.
We first treat the case in which \(E\) has analytic boundary.
The starting point is the exterior conformal map and the associated Faber expansion.
This makes it possible to relate approximation problems on the level curves \(L_R\) to the corresponding problems on the circles \(T_R\).
The proof of Theorem~\ref{thm:first-main} is given in Subsection~\ref{subsec:proof-first-main-general}.
We then prove Theorem~\ref{thm:two} first in the analytic-boundary case and then pass to the general Jordan-boundary case.
The proofs of Theorems~\ref{thm:equality-case} and~\ref{thm:common-subsequence} are given
in the final two subsections.


\subsection{Conformal reduction and Faber expansion}
\label{subsec:conformal-reduction}

Let \(E\subset\C\) be a continuum with at least two points and connected complement.
Let
\[
\Phi:\Omega=\widehat{\C}\setminus E\to\{w\in\widehat{\C}:|w|>1\}
\]
be the conformal map normalized by
\[
\Phi(\infty)=\infty,
\qquad
\Phi'(\infty)>0,
\]
and let
\[
\Psi:=\Phi^{-1}.
\]
Then
\begin{equation}\label{eq:green-phi}
g_\Omega(z,\infty)=\log |\Phi(z)|,
\qquad z\in\Omega.
\end{equation}
Here \(g_\Omega(\cdot,\infty)\) denotes the Green function of \(\Omega\) with pole at
\(\infty\).

For \(R\ge 1\), we recall
\[
E_R:=\{z\in\C:\ g_\Omega(z,\infty)\le \log R\},
\]
and for \(R>1\),
\[
L_R:=\{z\in\Omega:\ g_\Omega(z,\infty)=\log R\}.
\]
Since \eqref{eq:green-phi}, we also have
\[
L_R=\{z\in\Omega:|\Phi(z)|=R\},
\qquad R>1.
\]
For \(R>1\), we set
\[
\Omega_R:=\widehat{\C}\setminus E_R=\{z\in\Omega:|\Phi(z)|>R\},
\qquad
A_R:=\{w\in\widehat{\C}:|w|>R\}.
\]
Thus \(A_R\) and \(\Omega_R\) are the exterior domains containing \(\infty\)
bounded by \(T_R\) and \(L_R\), respectively.
The map \(\Psi\) sends \(T_R\) onto \(L_R\) and \(A_R\) conformally onto \(\Omega_R\).

We now assume in addition that \(E\) has analytic boundary. Then \(\partial E\) is an analytic Jordan curve, and \(\Psi\) extends analytically and univalently to a neighborhood of the unit circle. Hence there exists \(r_0<1\) such that \(\Psi\) is analytic and univalent in
\[
\{w\in\C:|w|>r_0\};
\]
see \cite{Saff1971}.

Let \(F_n\) be the Faber polynomials associated with \(E\). Since \(f\) is analytic
in a neighborhood of \(E\), it admits a Faber expansion
\[
f(z)=\sum_{n=0}^{\infty} a_n F_n(z),
\]
which converges in a neighborhood of \(E\).
Passing to the \(w\)-plane via \(z=\Psi(w)\), we obtain
\[
f(\Psi(w))
=
\sum_{n=0}^{\infty} a_n F_n(\Psi(w)).
\]
We write
\[
f(\Psi(w))=g(w)+h(w),
\]
where
\[
g(w):=\sum_{n=0}^{\infty} a_n w^n
\]
and
\[
h(w):=f(\Psi(w))-g(w).
\]

By the standard properties of Faber polynomials, the function \(g\) is analytic in
\(D_{R_0}\), while the function \(h\) is analytic in 
\[
\{w\in\C:|w|>r_0\}
\]
and satisfies
\[
h(\infty)=0.
\]
Thus the meromorphic continuation and approximation properties of \(f\) reduce
to those of \(g\), up to the analytic term \(h\). In particular, the meromorphy
radii \(R_k\) introduced for \(f\) coincide with the corresponding meromorphy
radii for \(g\) in the disc variable.

The notation introduced above will be used throughout this section.

\subsection{Meromorphic approximation on circles and level curves}
\label{subsec:meromorphic-comparison}

In this subsection we introduce the exterior meromorphic approximation problems
needed to pass from circles \(T_R\) to level curves \(L_R\). We keep the notation
\(\Delta\) from the bounded-domain setting considered earlier, where \(0\) lies in
the bounded complementary component. These are the same meromorphic approximation
problems as in that setting, but written in exterior form. They are reduced to the
bounded-domain case by the inversion \(z\mapsto 1/z\).

\medskip

\noindent
Let \(E\subset\C\) be a continuum with at least two points and connected complement.
For \(R>1\), let \(\mathcal M_{n,m}(A_R)\) denote the class of functions \(u\)
meromorphic in \(A_R\), having at most \(m\) finite poles in \(A_R\), counted with multiplicity,
and a pole at \(\infty\) of order at most \(n-m\), such that after subtracting
the principal parts of \(u\) at all its poles in \(A_R\cup\{\infty\}\), the remaining function
belongs to \(E_\infty(A_R)\).

\medskip

\noindent
Similarly, let \(\mathcal M_{n,m}(\Omega_R)\) denote the class of functions \(u\)
meromorphic in \(\Omega_R\), having at most \(m\) finite poles in \(\Omega_R\), counted with multiplicity,
and a pole at \(\infty\) of order at most \(n-m\), such that after subtracting
the principal parts of \(u\) at all its poles in \(\Omega_R\cup\{\infty\}\), the remaining function
belongs to \(E_\infty(\Omega_R)\).

\medskip

\noindent
We note that every rational function \(r\in \mathcal R_{n,m}\) belongs to both
\(\mathcal M_{n,m}(A_R)\) and \(\mathcal M_{n,m}(\Omega_R)\). Indeed, such a function is meromorphic
in the corresponding domain, has at most \(m\) finite poles, counted with multiplicity, has a pole
at \(\infty\) of order at most \(n-m\).

\medskip

\noindent
For a compact set \(K\), we write
\[
\|u\|_{L_\infty(K)}
\]
for the \(L_\infty\)-norm of \(u\) on \(K\).

\medskip

\noindent
For a function \(\varphi\) continuous on \(T_R\), define
\[
\Delta_{n,m}(\varphi;T_R)
:=
\inf\bigl\{
\|\varphi-u\|_{L_\infty(T_R)}:\ u\in\mathcal M_{n,m}(A_R)
\bigr\}.
\]
Similarly, for a function \(\varphi\) continuous on \(L_R\), define
\[
\Delta_{n,m}(\varphi;L_R)
:=
\inf\bigl\{
\|\varphi-u\|_{L_\infty(L_R)}:\ u\in\mathcal M_{n,m}(\Omega_R)
\bigr\}.
\]

\medskip
               
\noindent
Composition with \(\Psi\) or \(\Phi\) preserves the corresponding meromorphic class, and
addition or subtraction of a function in the relevant \(E_\infty\)-class does not change
the principal parts at the poles. Recall that
\[
f(\Psi(w))=g(w)+h(w).
\]
More precisely, the map
\[
u \mapsto u\circ\Psi-h
\]
gives a bijection between \(\mathcal M_{n,m}(\Omega_R)\) and \(\mathcal M_{n,m}(A_R)\), and
\[
\|g-(u\circ\Psi-h)\|_{L_\infty(T_R)}=\|f-u\|_{L_\infty(L_R)}.
\]
Similarly, the inverse correspondence is given by \(v\mapsto v\circ\Phi+h\circ\Phi\).
Hence the two meromorphic approximation problems are equivalent.
\begin{lemma}\label{lem:meromorphic-invariance}
For every \(n\ge m\ge 0\) and every \(R>1\), one has
\[
\Delta_{n,m}(f;L_R)=\Delta_{n,m}(g;T_R).
\]
\end{lemma}							
							
	\noindent
Thus the meromorphic approximation problem on the level curve \(L_R\) is reduced to
the corresponding problem on the circle \(T_R\). We next compare meromorphic and
rational approximation on the curves \(T_R\), \(L_R\) and on the associated compact
sets \(\overline{D}_R\), \(E_R\).						
\begin{lemma}\label{lem:curve-set-comparison}
For every \(n\ge m\ge 0\) and every \(R>1\), one has
\[
\Delta_{n,m}(g;T_R)\le \rho_{n,m}(g;\overline{D}_R)
\]
and
\[
\Delta_{n,m}(f;L_R)\le \rho_{n,m}(f;E_R).
\]
Moreover, for every \(1<\rho<R\), one has
\[
\rho_{n,m}(g;\overline{D}_\rho) \le c(\rho,R)\, \Delta_{n,m}(g;T_R)
\]
and
\[
\rho_{n,m}(f;E_\rho) \le C(\rho,R)\, \Delta_{n,m}(f;L_R),
\]
where \(c(\rho,R)>0\) and \(C(\rho,R)>0\) are independent of \(n\).
\end{lemma}
\begin{proof}
The first two estimates are immediate. Indeed, every rational function in
\(\mathcal R_{n,m}\) belongs both to \(\mathcal M_{n,m}(A_R)\) and to
\(\mathcal M_{n,m}(\Omega_R)\), and
\[
T_R\subset \overline{D}_R,\qquad L_R\subset E_R.
\]

\medskip
\noindent
We prove only
\[
\rho_{n,m}(f;E_\rho)\le C(\rho,R)\,\Delta_{n,m}(f;L_R),
\qquad 1<\rho<R,
\]
since the proof of
\[
\rho_{n,m}(g;\overline{D}_\rho)\le c(\rho,R)\,\Delta_{n,m}(g;T_R)
\]
is identical.

Let \(h\in \mathcal M_{n,m}(\Omega_R)\). Let \(r\in\mathcal R_{n,m}\) be the
rational function obtained by summing the principal parts of \(h\) at its finite
poles in \(\Omega_R\) and the polynomial part of \(h\) at \(\infty\), and set
\[
v:=h-r.
\]
Then \(v\) is holomorphic in \(\Omega_R\) and satisfies
\[
v(\infty)=0.
\]
Since
\[
f-r=(f-h)+v,
\]
and the exterior Cauchy integral of \(v\) vanishes, it follows that
\[
(f-r)(z)
=
\frac{1}{2\pi i}\int_{L_R}\frac{f(t)-h(t)}{t-z}\,dt,
\qquad z\in E_\rho.
\]
Therefore,
\[
\|f-r\|_{E_\rho}
\le
C(\rho,R)\,\|f-h\|_{L_\infty(L_R)},
\]
where \(C(\rho,R)>0\) is independent of \(n\)  and \(h\). Since \(r\in\mathcal R_{n,m}\),
it follows that
\[
\rho_{n,m}(f;E_\rho)
\le
\|f-r\|_{E_\rho}
\le
C(\rho,R)\,\|f-h\|_{L_\infty(L_R)}.
\]
Taking the infimum over all \(h\in\mathcal M_{n,m}(\Omega_R)\), we obtain
\[
\rho_{n,m}(f;E_\rho)\le C(\rho,R)\,\Delta_{n,m}(f;L_R).
\]
This proves the lemma.
\end{proof}
\begin{remark}\label{rem:curve-set-r0}
We now assume in addition that \(E\) has analytic boundary.
By Subsection~\ref{subsec:conformal-reduction}, there exists \(r_0\in(0,1)\)
such that \(\Psi\) is analytic and univalent in \(\{w\in\C:|w|>r_0\}\).
Moreover, \(g\) is analytic in \(D_{R_0}\). Hence the arguments of
Lemmas~\ref{lem:meromorphic-invariance} and \ref{lem:curve-set-comparison}
extend to the case \(R=1\).

More precisely, for every \(n\ge m\ge 0\), one has
\[
\Delta_{n,m}(f;\partial E)=\Delta_{n,m}(g;\T).
\]
Also, for every \(r\) with \(r_0<r<1\), one has
\[
\rho_{n,m}(g;\overline{D}_r)\le c(r,1)\,\Delta_{n,m}(g;\T),
\]
where \(c(r,1)>0\) is independent of \(n\). These \(R=1\) relations are
used in the proof of Theorem~\ref{thm:two} for analytic boundary.
\end{remark}


\subsection{Proof of Theorem~\ref{thm:first-main}}
\label{subsec:proof-first-main-general}

We prove Theorem~\ref{thm:first-main} for a continuum \(E\subset\C\) with at least two points and connected complement. 

The disc-case form of Theorem~\ref{thm:first-main} is invariant under dilations. Hence the statement for the pair \((\overline{\D},\overline{D}_R)\) immediately yields the corresponding statement for arbitrary pairs \((\overline{D}_r,\overline{D}_s)\), \(1<r<s\).

Fix \(m\ge0\) and \(R\) with \(1<R<R_0\). Choose numbers \(r_1,r_2,\rho\) such that
\[
1<r_1<r_2<\rho<R.
\]
Then, for each \(k=0,1,\dots,m\), we have
\[
\rho_{n-m+k,k}(f;E)
\le
\rho_{n-m+k,k}(f;E_{r_1}),
\]
since \(E\subset E_{r_1}\). Next, by Lemma~\ref{lem:curve-set-comparison},
\[
\rho_{n-m+k,k}(f;E_{r_1})
\le
C_k(r_1,r_2)\,\Delta_{n-m+k,k}(f;L_{r_2}),
\]
where \(C_k(r_1,r_2)>0\) is independent of \(n\). By
Lemma~\ref{lem:meromorphic-invariance},
\[
\Delta_{n-m+k,k}(f;L_{r_2})
=
\Delta_{n-m+k,k}(g;T_{r_2}).
\]
Finally, applying Lemma~\ref{lem:curve-set-comparison} in the circle case, we obtain
\[
\Delta_{n-m+k,k}(g;T_{r_2})
\le
\rho_{n-m+k,k}(g;\overline{D}_{r_2}).
\]
Combining these inequalities, we arrive at
\begin{equation}\label{eq:num-chain-k}
\rho_{n-m+k,k}(f;E)
\le
C_k(r_1,r_2)\,\rho_{n-m+k,k}(g;\overline{D}_{r_2}).
\end{equation}
Multiplying \eqref{eq:num-chain-k} over \(k=0,1,\dots,m\), we obtain
\begin{equation}\label{eq:num-chain-product}
\prod_{k=0}^{m}\rho_{n-m+k,k}(f;E)
\le
C(r_1,r_2)\,
\prod_{k=0}^{m}\rho_{n-m+k,k}(g;\overline{D}_{r_2}),
\end{equation}
where \(C(r_1,r_2)>0\) is independent of \(n\).

On the other hand, for each \(k=0,1,\dots,m\), Lemma~\ref{lem:curve-set-comparison}
gives
\[
\rho_{n-m+k,k}(f;E_R)
\ge
\Delta_{n-m+k,k}(f;L_R).
\]
By Lemma~\ref{lem:meromorphic-invariance},
\[
\Delta_{n-m+k,k}(f;L_R)
=
\Delta_{n-m+k,k}(g;T_R).
\]
Applying Lemma~\ref{lem:curve-set-comparison} in the circle case, we obtain
\[
\Delta_{n-m+k,k}(g;T_R)
\ge
c_k(\rho,R)\,\rho_{n-m+k,k}(g;\overline{D}_\rho),
\]
where \(c_k(\rho,R)>0\) is independent of \(n\). Hence
\begin{equation}\label{eq:den-chain-k}
\rho_{n-m+k,k}(f;E_R)
\ge
c_k(\rho,R)\,\rho_{n-m+k,k}(g;\overline{D}_\rho).
\end{equation}
Multiplying \eqref{eq:den-chain-k} over \(k=0,1,\dots,m\), we get
\begin{equation}\label{eq:den-chain-product}
\prod_{k=0}^{m}\rho_{n-m+k,k}(f;E_R)
\ge
c(\rho,R)\,
\prod_{k=0}^{m}\rho_{n-m+k,k}(g;\overline{D}_\rho),
\end{equation}
where \(c(\rho,R)>0\) is independent of \(n\).
\medskip

\noindent
Dividing \eqref{eq:num-chain-product} by \eqref{eq:den-chain-product}, we conclude that
\[
\prod_{k=0}^{m}
\frac{\rho_{n-m+k,k}(f;E)}
{\rho_{n-m+k,k}(f;E_R)}
\le
C(r_1,r_2,\rho,R)\,
\prod_{k=0}^{m}
\frac{\rho_{n-m+k,k}(g;\overline{D}_{r_2})}
{\rho_{n-m+k,k}(g;\overline{D}_\rho)}.
\]
where \(C(r_1,r_2,\rho,R)>0\) is independent of \(n\).
\medskip

Applying the disc-case form of Theorem~\ref{thm:first-main} to \(g\) on the pair
\((\overline{D}_{r_2},\overline{D}_\rho)\), we obtain
\[
\limsup_{n\to\infty}
\left(
\prod_{k=0}^{m}
\frac{\rho_{n-m+k,k}(g;\overline{D}_{r_2})}
{\rho_{n-m+k,k}(g;\overline{D}_\rho)}
\right)^{1/n}
\le
\frac{r_2^{m+1}}{\rho^{m+1}}.
\]
Hence
\[
\limsup_{n\to\infty}
\left(
\prod_{k=0}^{m}
\frac{\rho_{n-m+k,k}(f;E)}
{\rho_{n-m+k,k}(f;E_R)}
\right)^{1/n}
\le
\frac{r_2^{m+1}}{\rho^{m+1}}.
\]
Since \(1<r_2<\rho<R\) were arbitrary, letting \(r_2\downarrow1\) and
\(\rho\uparrow R\) yields
\[
\limsup_{n\to\infty}
\left(
\prod_{k=0}^{m}
\frac{\rho_{n-m+k,k}(f;E)}
{\rho_{n-m+k,k}(f;E_R)}
\right)^{1/n}
\le
\frac{1}{R^{m+1}}.
\]
This proves Theorem~\ref{thm:first-main}.


\subsection{Proof of Theorem~\ref{thm:two} for analytic boundary}
\label{subsec:proof-analytic-boundary}

We now prove Theorem~\ref{thm:two} in the case where \(E\) is a continuum with
connected complement and analytic boundary.

Fix \(m\ge0\) and \(R\) with \(1\le R<R_0\). Recall from
Subsection~\ref{subsec:conformal-reduction} that
\[
f(\Psi(w))=g(w)+h(w),
\]
where \(g\) is analytic in \(D_{R_0}\), while \(h\) is analytic in
\(\{w\in\C:\ |w|>r_0\}\) for some \(r_0<1\) and satisfies \(h(\infty)=0\).
Since \(h\) is analytic in \(\{w\in\C:\ |w|>r_0\}\), the functions \(f\) and \(g\)
have the same meromorphic continuation and rational approximation behavior. In
particular, the quantities \(R_k\) appearing in Theorem~\ref{thm:two} are
precisely the corresponding meromorphy radii for \(g\) in the disc variable.

The upper estimate
\begin{equation}\label{eq:analytic-upper}
\limsup_{n\to\infty}
\left(
\prod_{k=0}^{m}\rho_{n-m+k,k}(f;E_R)
\right)^{1/n}
\le
\frac{R^{m+1}}{R_0R_1\cdots R_m}
\end{equation}
follows by applying the Saff--Gonchar theorem to each factor and multiplying the resulting
estimates. We may assume in what follows that \(R_m<\infty\), since otherwise
the right-hand side of \eqref{eq:analytic-upper} is equal to \(0\), and the theorem
follows immediately. It remains to prove the reverse inequality.

We distinguish two cases.

\smallskip
\noindent
\textit{Case 1: \(1<R<R_0\).}
For each \(k=0,1,\dots,m\), we have
\[
\rho_{n-m+k,k}(f;E_R)
\ge
\rho_{n-m+k,k}(f;L_R)
\ge
\Delta_{n-m+k,k}(f;L_R).
\]
By Lemma~\ref{lem:meromorphic-invariance},
\[
\Delta_{n-m+k,k}(f;L_R)=\Delta_{n-m+k,k}(g;T_R).
\]
Fix \(r\) with \(1<r<R\). Then Lemma~\ref{lem:curve-set-comparison} yields
\[
\Delta_{n-m+k,k}(g;T_R)
\ge
c_k(r,R)\,\rho_{n-m+k,k}(g;\overline{D}_r),
\qquad k=0,1,\dots,m.
\]
Multiplying over \(k=0,1,\dots,m\), we obtain
\[
\prod_{k=0}^{m}\rho_{n-m+k,k}(f;E_R)
\ge
C(r,R)\prod_{k=0}^{m}\rho_{n-m+k,k}(g;\overline{D}_r),
\]
where \(C(r,R)>0\) is independent of \(n\). Taking \(n\)th roots, passing to the
limit superior, and using the disc-case formula for \(g\) on \(\overline{D}_r\), we obtain
\[
\limsup_{n\to\infty}
\left(
\prod_{k=0}^{m}\rho_{n-m+k,k}(f;E_R)
\right)^{1/n}
\ge
\frac{r^{m+1}}{R_0R_1\cdots R_m}.
\]
Letting \(r\uparrow R\), we conclude that
\begin{equation}\label{eq:analytic-lower-R>1}
\limsup_{n\to\infty}
\left(
\prod_{k=0}^{m}\rho_{n-m+k,k}(f;E_R)
\right)^{1/n}
\ge
\frac{R^{m+1}}{R_0R_1\cdots R_m}.
\end{equation}

\smallskip
\noindent
\textit{Case 2: \(R=1\).}
 For each \(k=0,1,\dots,m\), we have
\[
\rho_{n-m+k,k}(f;E)
\ge
\rho_{n-m+k,k}(f;\partial E)
\ge
\Delta_{n-m+k,k}(f;\partial E).
\]
By Remark~\ref{rem:curve-set-r0}, the \(R=1\) versions of
Lemmas~\ref{lem:meromorphic-invariance} and
\ref{lem:curve-set-comparison} hold. Hence, for each
\(k=0,1,\dots,m\),
\[
\Delta_{n-m+k,k}(f;\partial E)=\Delta_{n-m+k,k}(g;\T).
\]
Fix \(r\) with \(r_0<r<1\). Then
\[
\Delta_{n-m+k,k}(g;\T)
\ge
c_k(r,1)\,\rho_{n-m+k,k}(g;\overline{D}_r),
\qquad k=0,1,\dots,m.
\]
Multiplying over \(k=0,1,\dots,m\), we obtain
\[
\prod_{k=0}^{m}\rho_{n-m+k,k}(f;E)
\ge
C(r)\prod_{k=0}^{m}\rho_{n-m+k,k}(g;\overline{D}_r),
\]
where \(C(r)>0\) is independent of \(n\). Taking \(n\)th roots, passing to the
limit superior, and using the disc-case formula for \(g\) on \(\overline{D}_r\), we obtain
\[
\limsup_{n\to\infty}
\left(
\prod_{k=0}^{m}\rho_{n-m+k,k}(f;E)
\right)^{1/n}
\ge
\frac{r^{m+1}}{R_0R_1\cdots R_m}.
\]
Letting \(r\uparrow 1\), we obtain
\begin{equation}\label{eq:analytic-lower-R=1}
\limsup_{n\to\infty}
\left(
\prod_{k=0}^{m}\rho_{n-m+k,k}(f;E)
\right)^{1/n}
\ge
\frac{1}{R_0R_1\cdots R_m}.
\end{equation}
Combining \eqref{eq:analytic-upper} with \eqref{eq:analytic-lower-R>1} in Case~1
and with \eqref{eq:analytic-lower-R=1} in Case~2, we conclude that for every
\(R\) with \(1\le R<R_0\),
\[
\limsup_{n\to\infty}
\left(
\prod_{k=0}^{m}\rho_{n-m+k,k}(f;E_R)
\right)^{1/n}
=
\frac{R^{m+1}}{R_0R_1\cdots R_m}.
\]
This proves Theorem~\ref{thm:two} in the analytic-boundary case.


\subsection{Proof of Theorem~\ref{thm:two} for Jordan boundary}
\label{subsec:proof-jordan-boundary}

We now pass from the analytic-boundary case to the case where \(E\) is a
continuum with connected complement and Jordan boundary. It is enough to prove
Theorem~\ref{thm:two} for \(R=1\). Indeed, if \(1<R<R_0\), then the sublevel set
\(E_R\) has analytic Jordan boundary, and the desired formula follows from
Subsection~\ref{subsec:proof-analytic-boundary} applied to the compact set \(E_R\).

Thus it remains to consider the case \(R=1\), that is, the continuum \(E\)
itself. For this, we follow the final step in Saff's proof of his  original theorem
 for continua bounded by a Jordan curve \cite{Saff1971},
namely, the passage from the analytic-boundary case to the Jordan-boundary case
based on a classical approximation theorem of Walsh; see
\cite[\S2.1]{Walsh1969}.

Recall that \(g_{\Omega}(z,\infty)\) denotes the Green function of
\[
\Omega:=\widehat{\C}\setminus E
\]
with pole at \(\infty\), extended by zero on \(E\).

\noindent
The upper estimate
\begin{equation}\label{eq:jordan-upper-R1}
\limsup_{n\to\infty}
\left(
\prod_{k=0}^{m}\rho_{n-m+k,k}(f;E)
\right)^{1/n}
\le
\frac{1}{R_0R_1\cdots R_m}
\end{equation}
follows by applying the Saff--Gonchar theorem to each factor and multiplying the resulting
estimates. It remains to prove the reverse inequality. The following lemma provides the
geometric comparison needed for this step.

\begin{lemma}\label{lem:jordan-approximation}
Fix \(r>1\). Then there exists a continuum \(E'\) with analytic Jordan boundary such
that
\[
E'\subset E\subset E'_r,
\]
where
\[
\Omega':=\widehat{\C}\setminus E',
\qquad
E'_r:=\{z\in\C:\ g_{\Omega'}(z,\infty)\le \log r\},
\]
and \(g_{\Omega'}(z,\infty)\) denotes the Green function of \(\Omega'\) with pole at
\(\infty\), extended by zero on \(E'\).

Moreover, for every \(\rho>1\),
\[
E'_\rho\subset E_\rho\subset E'_{r\rho},
\]
and
\[
R_k\le R_k'\le rR_k,
\qquad k=0,1,\dots,m,
\]
where \(R_k'\) denotes the \(k\)th radius of meromorphic continuation of
\(f\) relative to \(E'\).
\end{lemma}
\begin{proof}
By the approximation step used by Saff in \cite{Saff1971}, based on Walsh's
theorem \cite[\S2.1]{Walsh1969}, there exists a continuum \(E'\) with analytic
Jordan boundary such that
\[
E'\subset E\subset E'_r.
\]
Since \(E'\subset E\), one has \(\Omega\subset \Omega'\). Hence
\[
g_{\Omega'}(z,\infty)\ge g_{\Omega}(z,\infty),
\qquad z\in\Omega,
\]
and therefore
\[
E'_\rho\subset E_\rho,
\qquad \rho>1.
\]

Set
\[
u(z):=g_{\Omega'}(z,\infty)-g_{\Omega}(z,\infty),
\qquad z\in\Omega.
\]
Then \(u\) is harmonic and nonnegative in \(\Omega\). We show that
\[
0\le u(z)\le \log r,
\qquad z\in\Omega.
\]
Indeed,
since
\[
E\subset E'_r
=
\{z\in\C:\ g_{\Omega'}(z,\infty)\le \log r\},
\]
we have
\[
u(z)=g_{\Omega'}(z,\infty)\le \log r,
\qquad z\in\partial E,
\]
because \(g_{\Omega}(z,\infty)=0\) on \(E\). Hence, by the maximum principle,
\[
0\le u(z)\le \log r,
\qquad z\in\Omega.
\]

Now let \(\rho>1\). If \(z\in E_\rho\), then
\[
g_{\Omega}(z,\infty)\le \log\rho.
\]
Combining this with the estimate \(u(z)\le \log r\), we obtain
\[
g_{\Omega'}(z,\infty)
=
g_{\Omega}(z,\infty)+u(z)
\le
\log\rho+\log r
=
\log(r\rho).
\]
Therefore,
\[
E_\rho\subset E'_{r\rho},
\qquad \rho>1.
\]

Thus
\[
E'_\rho\subset E_\rho\subset E'_{r\rho},
\qquad \rho>1.
\]
These inclusions imply 
\[
R_k\le R_k'\le rR_k,
\qquad k=0,1,\dots,m.
\]
\end{proof}

\smallskip

Fix \(r>1\), and let \(E'\) be as in Lemma~\ref{lem:jordan-approximation}. Since
\(E'\subset E\), for each \(k=0,1,\dots,m\),
\[
\rho_{n-m+k,k}(f;E')
\le
\rho_{n-m+k,k}(f;E).
\]
Multiplying over \(k=0,1,\dots,m\), we obtain
\[
\prod_{k=0}^{m}\rho_{n-m+k,k}(f;E')
\le
\prod_{k=0}^{m}\rho_{n-m+k,k}(f;E).
\]
Therefore,
\[
\limsup_{n\to\infty}
\left(
\prod_{k=0}^{m}\rho_{n-m+k,k}(f;E)
\right)^{1/n}
\ge
\limsup_{n\to\infty}
\left(
\prod_{k=0}^{m}\rho_{n-m+k,k}(f;E')
\right)^{1/n}.
\]
Since \(E'\) has analytic Jordan boundary, the analytic-boundary case proved above
applies to \(E'\). Hence
\[
\limsup_{n\to\infty}
\left(
\prod_{k=0}^{m}\rho_{n-m+k,k}(f;E')
\right)^{1/n}
=
\frac{1}{R_0'R_1'\cdots R_m'}.
\]
Using \(R_k'\le rR_k\), we get
\[
\frac{1}{R_0'R_1'\cdots R_m'}
\ge
\frac{1}{r^{m+1}R_0R_1\cdots R_m}.
\]
Consequently,
\[
\limsup_{n\to\infty}
\left(
\prod_{k=0}^{m}\rho_{n-m+k,k}(f;E)
\right)^{1/n}
\ge
\frac{1}{r^{m+1}R_0R_1\cdots R_m}.
\]
Letting \(r\downarrow1\), we obtain
\begin{equation}\label{eq:jordan-lower-R1}
\limsup_{n\to\infty}
\left(
\prod_{k=0}^{m}\rho_{n-m+k,k}(f;E)
\right)^{1/n}
\ge
\frac{1}{R_0R_1\cdots R_m}.
\end{equation}
Combining \eqref{eq:jordan-upper-R1} and \eqref{eq:jordan-lower-R1}, we conclude that
\[
\limsup_{n\to\infty}
\left(
\prod_{k=0}^{m}\rho_{n-m+k,k}(f;E)
\right)^{1/n}
=
\frac{1}{R_0R_1\cdots R_m}.
\]
This proves Theorem~\ref{thm:two}.


\subsection{Proof of Theorem~\ref{thm:equality-case}}

We now combine Theorems~\ref{thm:first-main}
and~\ref{thm:two} to prove the equality case.

Fix \(R\) with \(1<R<R_0\). Applying Theorem~\ref{thm:two}, we obtain
\[
\limsup_{n\to\infty}
\left(
\prod_{k=0}^{m}\rho_{n-m+k,k}(f;E_R)
\right)^{1/n}
=
\frac{R^{m+1}}{R_0R_1\cdots R_m},
\]
while Theorem~\ref{thm:two} with \(R=1\) gives
\[
\limsup_{n\to\infty}
\left(
\prod_{k=0}^{m}\rho_{n-m+k,k}(f;E)
\right)^{1/n}
=
\frac{1}{R_0R_1\cdots R_m}.
\]
Therefore, since
\[
\limsup_{n\to\infty}\frac{a_n}{b_n}
\ge
\frac{\limsup_{n\to\infty}a_n}{\limsup_{n\to\infty}b_n},
\qquad a_n>0,\ b_n>0,
\]
it follows that
\[
\limsup_{n\to\infty}
\left(
\prod_{k=0}^{m}
\frac{\rho_{n-m+k,k}(f;E)}
{\rho_{n-m+k,k}(f;E_R)}
\right)^{1/n}
\ge
\frac{1}{R^{m+1}}.
\]
Combining this lower bound with the upper bound from Theorem~\ref{thm:first-main}, we conclude that
\[
\limsup_{n\to\infty}
\left(
\prod_{k=0}^{m}
\frac{\rho_{n-m+k,k}(f;E)}
{\rho_{n-m+k,k}(f;E_R)}
\right)^{1/n}
=
\frac{1}{R^{m+1}}.
\]
This proves Theorem~\ref{thm:equality-case}.
\subsection{A common subsequence for rows and partial products}
\label{subsec:common-subsequence}

We now prove Theorem~\ref{thm:common-subsequence}. We begin by showing that
for each fixed \(m\), the extremal exponential behavior in the Saff--Gonchar theorem is
attained simultaneously for all rows \(k=0,1,\dots,m\), and likewise for the
corresponding partial products.

Fix \(m\ge 0\) and assume that \(R_m<\infty\).
By Theorem~\ref{thm:two} with \(R=1\), we have
\[
\limsup_{n\to\infty}
\left(\prod_{j=0}^{m}\rho_{n-m+j,j}(f;E)\right)^{1/n}
=
\frac{1}{R_0R_1\cdots R_m}.
\]
Choose a subsequence \(\Lambda\subset\mathbb N\) such that
\[
\left(\prod_{j=0}^{m}\rho_{n-m+j,j}(f;E)\right)^{1/n}
\to
\frac{1}{R_0R_1\cdots R_m},
\qquad n\to\infty,\ \ n\in\Lambda.
\]

For each \(j=0,1,\dots,m\), the Saff--Gonchar theorem gives
\[
\limsup_{n\to\infty}\rho_{n-m+j,j}(f;E)^{1/n}=\frac{1}{R_j}.
\]
Hence, along the subsequence \(\Lambda\),
\[
\limsup_{\substack{n\to\infty\\ n\in\Lambda}}
\rho_{n-m+j,j}(f;E)^{1/n}\le \frac{1}{R_j},
\qquad j=0,1,\dots,m.
\]
Since
\[
\prod_{j=0}^{m}\rho_{n-m+j,j}(f;E)^{1/n}
\to
\prod_{j=0}^{m}\frac{1}{R_j},
\qquad n\to\infty,\ \ n\in\Lambda,
\]
it follows that each factor converges to its upper bound, that is,
\[
\rho_{n-m+j,j}(f;E)^{1/n}\to \frac{1}{R_j},
\qquad j=0,1,\dots,m,
\qquad n\to\infty,\ \ n\in\Lambda.
\]
Now fix \(k\) with \(0\le k\le m\). Taking the product over \(j=0,1,\dots,k\), we obtain
\[
\left(\prod_{j=0}^{k}\rho_{n-m+j,j}(f;E)\right)^{1/n}
\to
\frac{1}{R_0R_1\cdots R_k},
\qquad n\to\infty,\ \ n\in\Lambda.
\]
This proves the desired assertions on \(E\).

Now let \(R\) satisfy \(1<R<R_0\). We claim that
\begin{equation}\label{eq:ER-product-subsequence}
\left(\prod_{j=0}^{m}\rho_{n-m+j,j}(f;E_R)\right)^{1/n}
\to
\frac{R^{m+1}}{R_0R_1\cdots R_m},
\qquad n\to\infty,\ \ n\in\Lambda.
\end{equation}
Indeed, if this were false, then, since by Theorem~\ref{thm:two}
\[
\limsup_{n\to\infty}
\left(\prod_{j=0}^{m}\rho_{n-m+j,j}(f;E_R)\right)^{1/n}
=
\frac{R^{m+1}}{R_0R_1\cdots R_m},
\]
there would exist a subsequence \(\Lambda'\subset\Lambda\) such that
\[
\left(\prod_{j=0}^{m}\rho_{n-m+j,j}(f;E_R)\right)^{1/n}
\to \beta
\qquad \text{as } n\to\infty,\ \ n\in\Lambda',
\]
with
\[
\beta<\frac{R^{m+1}}{R_0R_1\cdots R_m}.
\]
Since, by the choice of \(\Lambda\),
\[
\left(\prod_{j=0}^{m}\rho_{n-m+j,j}(f;E)\right)^{1/n}
\to
\frac{1}{R_0R_1\cdots R_m},
\qquad n\to\infty,\ \ n\in\Lambda,
\]
it would then follow that
\[
\left(
\prod_{j=0}^{m}
\frac{\rho_{n-m+j,j}(f;E)}
{\rho_{n-m+j,j}(f;E_R)}
\right)^{1/n}
\to
\frac{\displaystyle \frac{1}{R_0R_1\cdots R_m}}{\beta}
>
\frac{1}{R^{m+1}}.
\]
as \(n\to\infty\), \(n\in\Lambda'\), contrary to Theorem~\ref{thm:equality-case}.
This proves the claim.

For each \(j=0,1,\dots,m\), the Saff--Gonchar theorem on the set \(E_R\) gives
\[
\limsup_{n\to\infty}\rho_{n-m+j,j}(f;E_R)^{1/n}=\frac{R}{R_j}.
\]
Arguing exactly as above, and using \eqref{eq:ER-product-subsequence}, we conclude that along \(\Lambda\),
\[
\rho_{n-m+j,j}(f;E_R)^{1/n}\to \frac{R}{R_j},
\qquad j=0,1,\dots,m,
\qquad n\to\infty,\ \ n\in\Lambda.
\]
Therefore, for each \(k=0,1,\dots,m\),
\[
\left(\prod_{j=0}^{k}\rho_{n-m+j,j}(f;E_R)\right)^{1/n}
\to
\frac{R^{k+1}}{R_0R_1\cdots R_k},
\qquad n\to\infty,\ \ n\in\Lambda.
\]
and
\[
\left(
\prod_{j=0}^{k}
\frac{\rho_{n-m+j,j}(f;E)}
{\rho_{n-m+j,j}(f;E_R)}
\right)^{1/n}
\to
\frac{1}{R^{k+1}},
\qquad n\to\infty,\ \ n\in\Lambda.
\]
This proves Theorem~\ref{thm:common-subsequence}.



\noindent
Vasiliy A. Prokhorov\\
Department of Mathematics and Statistics\\
University of South Alabama\\
Mobile, Alabama 36688-0002\\
{\tt prokhoro@southalabama.edu}
\bigskip

\end{document}